\documentclass[12pt,reqno]{amsart}
\usepackage{amssymb}
\usepackage{geometry}
\usepackage[latin1]{inputenc}
\usepackage[italian,english]{babel}
\usepackage{amsmath, amsfonts, amsthm}
\usepackage{graphicx}
\usepackage{pgfplots}
\usepackage{paralist}
\usepackage{mathtools, mathabx,bigints}
\numberwithin{equation}{section}
\newtheorem{thm}{Theorem}[section]
\newtheorem{lem}[thm]{Lemma}
\newtheorem{cor}[thm]{Corollary}
\newtheorem{prop}[thm]{Proposition}

\theoremstyle{definition}

\theoremstyle{remark}

\newcommand{\R}{\mathbb{R}}
\newcommand{\N}{\mathbb{N}}

\newcommand{\argmin}{\arg\min}

\DeclareMathOperator{\dive}{div}

\DeclareMathOperator{\sign}{sign}

\DeclareMathOperator{\curl}{curl}
\newenvironment{sistema}%
{\left\{\begin{array}{@{}l@{}}}{\end{array}\right.}
\patchcmd{\abstract}{\scshape\abstractname}{\textbf{\abstractname}}{}{}
\makeatletter 
\def\@makefnmark{} 
\makeatother 
 
\begin{document}
\title[Monotonicity Principle in Tomography of Nonlinear Materials]{Monotonicity Principle in Tomography of Nonlinear Conducting Materials}
\author[A. Corbo Esposito, L. Faella, G. Piscitelli, R. Prakash, A. Tamburrino]{
Antonio Corbo Esposito$^1$, Luisa Faella$^1$, Gianpaolo Piscitelli$^1$, Ravi Prakash$^2$, Antonello Tamburrino$^{1,3}$}\footnote{\\$^1$Dipartimento di Ingegneria Elettrica e dell'Informazione \lq\lq M. Scarano\rq\rq, Universit\`a degli Studi di Cassino e del Lazio Meridionale, Via G. Di Biasio n. 43, 03043 Cassino (FR), Italy.\\
$^2$Departamento de Matem\'atica, Facultad de Ciencias F\'isicas y Matem\'aticas, Universidad de Concepci\'on, Avenida Esteban Iturra s/n, Bairro Universitario, Casilla 160 C, Concepci\'on, Chile.\\
$^3$Department of Electrical and Computer Engineering, Michigan State University, East Lansing, MI-48824, USA.\\
Email: antonio.corboesposito@unicas.it, luisa.faella@unicas.it, gianpaolo.piscitelli@unicas.it, rprakash@udec.cl, antonello.tamburrino@unicas.it {\it (corresponding author)}.}
\maketitle
\begin{abstract} We treat an inverse electrical conductivity problem  which deals with the reconstruction of nonlinear electrical conductivity starting from boundary measurements in steady currents operations. In this framework, a key role is played by the Monotonicity Principle,  which establishes a monotonic relation connecting the unknown material property to the (measured) Dirichlet-to-Neumann operator (DtN).
Monotonicity Principles are the foundation for a class of non-iterative and real-time imaging methods and algorithms. 

In this article, we prove that the Monotonicity Principle for the Dirichlet Energy in nonlinear problems holds under mild assumptions. Then, we show that apart from linear and $p$-Laplacian cases, it is impossible to transfer this Monotonicity result from the Dirichlet Energy to the DtN operator. To overcome this issue, we introduce a new boundary operator, identified as an Average DtN operator.
 


\end{abstract}

\section{Introduction} 
In this paper, we derive the Monotonicity Principle  for an inverse conductivity problem modeled by a fully nonlinear variant of the Calder\'on problem. Specifically, we treat the problem of retrieving the nonlinear electrical conductivity $\sigma$ starting from boundary measurements for stationary cases (steady currents). More precisely, we consider a fully nonlinear problem where the constitutive relationship is local, isotropic and memoryless:  
\begin{equation} \label{J}
{\bf J}(x)=\sigma(x,\vert {\bf E}(x)\vert) {\bf E}(x)\quad\forall x\in\Omega,
\end{equation}
where $\sigma$ is the nonlinear electrical conductivity, ${\bf J}$ the electric current density, ${\bf E}$ the electric field and $\Omega\subset\R^n$, $n \geq 2,$ is an open bounded domain with Lipschitz boundary. $\Omega$ represents the region occupied by the conducting material.

In steady currents operations, the electric field can be expressed through the electrical scalar potential $u\in W^{1,p}(\Omega)$ as ${\bf E}(x)=-\nabla u(x)$, where $p$ depends on the behaviour of $\sigma$. The electric scalar potential $u$ solves the steady current problem:
\begin{equation}\label{gproblem1}
\begin{cases}
\dive\Big(\sigma (x, |\nabla u(x)|) \nabla u (x)\Big) =0\ \text{in }\Omega\vspace{0.2cm}\\
u(x) =f(x)\qquad\qquad\qquad\quad\  \text{on }\partial\Omega,
\end{cases}
\end{equation}
where $f\in X_\diamond$ is the applied boundary potential, with $X_\diamond$ being an appropriate abstract trace space (see Section \ref{FoP}). The existence of a solution is guaranteed under suitable assumptions on $\sigma$ (see Section \ref{FoP}). 

Nonlinear electrical conductivities can be found in semiconducting and ceramic materials (see \cite{bueno2008sno2}). Among the applications, cable termination in high voltage (HV) and medium voltage (MV) systems \cite{boucher2018interest, lupo1996field} are well studied examples. Nonlinear electrical conductivities is also found in superconductors, key materials for applications like energy storage, magnetic levitation systems, superconducting magnets (nuclear fusion devices, nuclear magnetic resonance) and high-frequency radio technology \cite{seidel2015applied, krabbes2006high}. In all these applications, one can expect the need for nondestructive evaluation and/or imaging. For example, the Authors of \cite{lee2005nde,takahashi2014non,amoros2012effective} treated the nondestructive testing in the presence of superconductors. Nonlinear models for the electrical conductivity can be found also in the area of biological tissues (see \cite{foster1989dielectric}). For instance, \cite{corovic2013modeling} proved that nonlinear models fit the experimental data better than linear models.

We stress that problem (\ref{gproblem1}) can be applied without any relevant modification to other physical settings. For instance, in the framework of electromagnetism, both nonlinear electrostatic and nonlinear magnetostatic\footnote{$^1$In magnetostatic, it is possible to introduce a magnetic scalar potential in simply connected and source free regions.}{$^1$} phenomena can be modeled in the form of (\ref{gproblem1}). The unknown material property is the dielectric permittivity for electrostatic case whereas the magnetic permeability for magnetostatic one. In the first case the constitutive relationship is ${\bf D}(x)=\varepsilon(x,\vert {\bf E}(x)\vert) {\bf E}(x)$ (see \cite{miga2011non} and references therein, see \cite{yarali20203d}), whereas in the second case ${\bf B}(x)=\mu(x,\vert {\bf H}(x)\vert) {\bf H}(x)$ (see \cite{1993ferr.book.....B}).

In this article, we develop a theoretical contribution to the field of inverse problems with nonlinear constitutive relationships. From a general perspective, as quoted in \cite{lam2020consistency}, {\it \lq\lq\ ... the mathematical analysis for inverse problems governed by nonlinear Maxwell's equations is still in the early stages of development.\rq\rq}. One can expect that as new methods and algorithms will be available, the demand for nondestructive evaluation and imaging of nonlinear materials will eventually and significantly arise.


In  this  framework, a key role is played by the (nonlinear) Dirichlet-to-Neumann (DtN) operator, mapping
the boundary voltage $f$ to the current flux at the boundary
\begin{equation*}
\Lambda _{\sigma }: f\in X_\diamond\rightarrow  -{\bf J}\cdot \mathbf{\hat{n}}\vert_{\partial\Omega}=\sigma\,
\partial_nu|_{\partial\Omega} 
\in X_\diamond'
\end{equation*}
where $u$ is the solution of \eqref{gproblem1} with boundary data $f$, $X_\diamond'$ is the dual of $X_\diamond$ and $\mathbf{\hat{n}}$ denotes the outer unit normal on $\partial\Omega$. 

The goal of this paper is to provide a \lq\lq tool\rq\rq, the Monotonicity Principle, to reconstruct the nonlinear electrical conductivity $\sigma$ starting from the knowledge of the boundary data $\Lambda_\sigma$.

The Monotonicity Principle Method (MPM) is an imaging method which relies on a monotone relation connecting the
unknown material property to the measured DtN or its inverse. In the linear case, MPM states that
\begin{equation}\label{mondef}
\sigma \leq \tau \implies \Lambda _{\sigma} \leq \Lambda _{\tau}.
\end{equation}
where $\sigma$ and $\tau$ are two electrical conductivities defined in $\Omega$, $\Lambda _{\sigma }$ and $\Lambda _{\tau}$ are the corresponding DtN operators. In equation (\ref{mondef}), $\sigma\leq \tau$ is understood in the almost everywhere sense in $\Omega$, and $\Lambda_\sigma\leq\Lambda_\tau$ means that $\Lambda_\sigma-\Lambda_\tau$ is negative semidefinite. 
Monotonicity relation (\ref{mondef})
shows that a pointwise increase of the electrical conductivity leads to \lq\lq greater\rq\rq\ boundary data.

Monotonicity (\ref{mondef}) is the basis to develop non-iterative and real-time reconstruction methods and algorithms \cite{Tamburrino_2002,Tamburrino2006FastMF,Tamburrino_2006}. 
MPM has been mainly applied to shape reconstruction problems for detecting the shape of anomalies in a given background. In this specialization, the method determines if a test inclusion is part of the anomaly or not by a simple test. Indeed, if  $T$ is a \lq\lq test\rq\rq\ anomaly and $V$ is the unknown anomaly, corresponding to the electrical conductivities given by $\sigma_T$ and $\sigma_V$, respectively,  (\ref{mondef}) implies 
\begin{equation}\label{mondef2}
T\subseteq V \implies \Lambda _{\sigma_T} \geq \Lambda _{\sigma_V},
\end{equation}
where we have assumed that both anomalies have electrical conductivity smaller than that of the background.
Equation (\ref{mondef2}) corresponds to
\begin{equation}\label{mondef3}
 \Lambda _{\sigma_T} \ngeqslant \Lambda _{\sigma_V}\implies T \nsubseteq V. 
\end{equation}
Therefore, from the knowledge of the boundary DtN operators, we can infer if the prescribed test anomaly $T$ is contained or not in the unknown anomaly $V$. By repeating the test in (\ref{mondef3}) for various sets $T$, we can reconstruct an approximation of the shape
of the unknown anomaly $V$.

According to our awareness, the first evidence of a monotone property (\ref{mondef}) for the linear case appeared in \cite{gisser1990electric}. Then its relevance to the field of Inverse Problems was first recognized in \cite{Tamburrino_2002}, where Tamburrino and Rubinacci proposed \eqref{mondef3} to establish a new imaging method. Specifically, they (i) proved the equivalent of \eqref{mondef} but for a real-world system made by a finite number of electrodes, (ii) proposed and numerically tested the imaging method based on \eqref{mondef3} and (iii) extended (\ref{mondef}) to perfectly conducting or insulating anomalies.

Surprisingly, Monotonicity Principles, appear to be a general feature which can be found in many problems governed by PDEs of different nature. Indeed, despite originally found for elliptic PDEs arising from static problems (as, for instance, Electrical Resistance, Capacitance or Inductance Tomography) \cite{Tamburrino_2002,Tamburrino_2006,Calvano201232,Tamburrino2003233,harrach2013monotonicity}, it was also found for elliptic PDEs but arising from quasi-static problems (as, for instance, Eddy Current Tomography) \cite{Tamburrino2006FastMF,Tamburrino_2006,Tamburrino_2010}. 

Parabolic PDEs (for instance, pulsed Eddy Current Tomography) have been treated in \cite{Su_2017,Tamburrino20161,Tamburrino2015159,Su2017,Tamburrino_2016_testing}. Specifically, it was proved a Monotonicity Principle for the time constants of the natural modes.

Monotonicity of the transmission eigenvalues for the Helmoltz equation was analyzed in \cite{AT_WAVE2015}. Other Monotonicity Principles for problems governed by the Helmholtz equation were developed in \cite{harrach2019dimension, harrach2019monotonicity} for bounded domains and in \cite{griesmaier2018monotonicity,Albicker_2020} for unbounded domains. Monotonicity was also applied to crack detection for the Helmholtz equation in \cite{daimon2020monotonicity}.

Monotonicity for linear elasticity was introduced in \cite{eberle2020shape}. 
 
A special feature of MPM is that it provides rigorous upper and lower bounds to the unknown, even in the presence of noise, under proper hypothesis \cite{Tamburrino2016284}.

The limiting case of perfectly insulating anomalies has been treated in \cite{Tamburrino_2002,candiani2019monotonicity} and the case of perfectly conducting anomalies in \cite{Tamburrino_2002}. 

The concept of regularization for MPM was introduced in \cite{Rubinacci20061179,garde2017convergence}. This is relevant because MPM is not based upon the minimization of an objective function, where regularization can be easily introduced by means of penalty terms. Also, Monotonicity has been used as regularizer in \cite{harrach2016enhancing}.

A first experimental validation of MPM for Eddy Current Tomography can be found in \cite{Tamburrino201226}.

Additional numerical aspects have been studied in \cite{garde2017convergence,harrach2015resolution, harrach2018monotonicity,garde2018comparison,  garde2019regularized}.
The stability of the method has been treated and proved in \cite{harrach2013monotonicity,eberle2020shape,candiani2019monotonicity,garde2017convergence,garde2018comparison, garde2019regularized, harrach2010exact,eberle2020lipschitz}. Numerical evidence can be found in \cite{Su_2017}, whereas experimental evidence in \cite{Tamburrino201226}.

Theoretical applications of the Monotonicity can be found in the proof of uniqueness results \cite{harrach2009uniqueness, harrach2012simultaneous,harrach2017local}. 

Other than soft-field Tomography and Nondestructive Testing, MPM has been applied to homogenization of materials \cite{7559815} and concrete rebars inspection \cite{DeMagistris2007389,Rubinacci2007333}.

Monotonicity was combined with frequency-difference
and ultrasound modulated Electrical Impedance Tomography measurements, to reduce the impact of modelling errors such those arising from electrode positions and the shape of the imaging domain \cite{harrach2015combining}.


The converse of Monotonicity (\ref{mondef2}), under the assumption that the unknown anomaly consists of union of non-contractible sets, was proved by Harrach and Ullrich in \cite{harrach2013monotonicity}. This result
is relevant because it states that MPM gives exact reconstruction for (union of) contractible anomalies, at least in the ideal setting when the measured boundary data is the DtN operator. Unfortunately, this is not the case for practical systems made of a finite number of electrodes, where implication (\ref{mondef2}) holds but not its converse (see \cite{Tamburrino_2002}).

Monotonicity, but in a different form, has been treated in \cite{alessandrini1989remark}.

Eventually, it is worth noting that imaging methods based on Monotonicity Principle fall in the class of non-iterative imaging methods. Colton and Kirsch introduced the first non-iterative approach named Linear Sampling Method (LSM) \cite{Colton_1996} followed by the Factorization Method (FM) proposed by Kirsch \cite{Kirsch_1998}. Ikehata proposed the Enclosure Method \cite{ikehata1999draw,Ikehata_2000} and Devaney applied MUSIC (MUltiple SIgnal
Classification), a well known algorithm in signal processing, as imaging method \cite{Devaney2000}. 

A special case for equation $(\ref{J})$ is when $\sigma(x,\vert {\bf E}(x)\vert)=\theta(x)\vert {\bf E}(x)\vert^{p-2}$. Then the relationship between the electrical current density ${\bf J}$ and the electric field ${\bf E}$ can be written as
\begin{equation}\label{pLap}
{\bf J}(x)=\theta(x)\vert {{\bf E}(x)}\vert^{p-2} {{\bf E}(x)},
\end{equation}
where $\theta\in L^{\infty}(\Omega)$ and $\theta(x) \geq c_0$ a.e. in $\Omega$
 for some positive constant $c_0$. This leads to the study of a steady current problem involving the $p$-Laplacian. Here, briefly, we give an overview of the main and most recent results concerning the nonlinear p-Laplace type model. 
 The inverse problem of Calder\'on was initially posed in the setting of the p-conductivity
equation by Salo and Zhong \cite{doi:10.1137/110838224} and, then, studied in \cite{brander2015enclosure,brander2016calderon} also. 
In \cite{doi:10.1137/110838224}, the authors proved boundary determination results under proper regularity assumptions on the conductivity and boundary of the domain. 
Moreover, they showed that to investigate Calder\'on-type problems for equations with
weak non-linearities (see \cite{sun2004inverse,Sun_2005}), one can use the G\^{a}teaux derivative of the DtN operator at constant boundary values. The method does not work  for  p-Laplace equation  
as proved in the Appendix of \cite{doi:10.1137/110838224}.
For these Calder\'on-type problems, the Monotonicity inequality was  proved in \cite{brander2015enclosure} (see Lemma 2.1).
Here the authors study the enclosure method for non linear equation. The enclosure method introduced by Ikehata uses complex geometrical optics (CGO) solution in place of point sources (see \cite{ikehata1999draw,Ikehata_2000}).

The authors in \cite{carstea2020recovery} have treated a nonlinear problem (linear plus a nonlinear term) which is a particular case of the model proposed in this article.

In \cite{brander2016calderon}, the author proved a boundary uniqueness result for the first order normal derivative of the conductivity.
In \cite{brander2018superconductive}, the authors extended the weak DtN operator to conductivities that include regions of zero or infinite conductivity. 

The extension from the isotropic case of \cite{brander2015enclosure} to the anisotropic one has been proposed in \cite{guo2016inverse}. The authors used the Monotonicity inequality to prove injectivity for the DtN operator under a Monotonicity assumption (see Theorem 2.1 and Lemma 2.2). Their results show injectivity in two dimensions for Lipschitz conductivities. In higher dimension cases, one of the conductivities is required to be close to a constant.

In \cite{brander2018monotonicity}, a Calder\'on problem for nonlinear $p$-Laplacian type equations is studied. In this paper, the authors show that Monotonicity based shape reconstruction methods (\cite{Tamburrino_2002,Tamburrino_2006}) work in the $p$-Laplacian case which allow them to find the complex 
 hull of the inclusion without any regularity or interface jump assumption. As a matter of fact, any regularity on jump properties for the inclusion is not required and they obtain this subset using both the Monotonicity and enclosure method.   
For properties of DtN operator with $\theta=1$ in \eqref{pLap},  we refer to Hauer \cite{hauer2015p}.

The original contribution of this paper consists in deriving a Monotonicity Principle in the fully nonlinear case. This result is not at all a trivial development of the previous ones. Indeed, we prove that, in general, the Dirichlet Energy $\mathbb{F}_\sigma
\left(  u^{f} \right)$ is the quantity being monotone with respect to the electrical conductivity (Theorem \ref{monoten}), that is: 
\begin{equation} \label{m_genergy}
\sigma_1\leq\sigma_2\quad\Longrightarrow\quad\mathbb F_{\sigma_1}(u_{1}^{f}) \leq \mathbb F_{\sigma_2} (u_{2}^{f}),
\end{equation}
where $u_{1}^{f}$ and $u_{2}^{f}$ are the solutions of (\ref{gproblem1}) corresponding to $\sigma_1$ and $\sigma_2$, respectively, with $f$ being the applied boundary voltage. In (\ref{m_genergy}), $\sigma_1\le\sigma_2$ means that
\begin{equation*}
\sigma_1(x,E)\leq \sigma_2(x,E) \quad \text{for a.e.}\ x\in\overline\Omega\ \text{and}\ \forall\ E>0.
\end{equation*}

Moreover, we show that Monotonicity can be easily transferred to the boundary DtN operator, but only for linear and $p$-Laplacian problems. In these cases, we demonstrate that the DtN power product $\langle\Lambda_\sigma(f),f\rangle$ is proportional to the Dirichlet Energy and, hence, the Monotonicity for the boundary DtN operator follows.


When the nonlinearity is more general, for instance, of polynomial type, the Dirichlet Energy is monotone with respect to the constitutive relationship but it is not proportional to the power product $\langle\Lambda_\sigma(f),f\rangle$ (see Subsection \ref{connection}). Therefore, Monotonicity cannot be transferred from an \lq\lq internal\rq\rq\ quantity such as the Dirichlet Energy to boundary data like $\Lambda_\sigma$. This is a major issue since in solving inverse problem, we do not have any access to internal quantities like the Dirichlet Energy, instead, we have access to data which can only be measured from the boundary. Therefore, for the general nonlinear case, we need to \lq\lq transfer\rq\rq\ the Monotonicity Principle from the Dirichlet Energy (internal quantity) to another proper boundary operator other than $\Lambda_\sigma$.

Specifically, we introduce a new nonlinear boundary operator $\overline{\Lambda} _{\sigma }$ which is monotonic with respect to the nonlinear material property. We name $\overline{\Lambda} _{\sigma }$ as the \textit{Average DtN Operator} which is defined as
\begin{equation*}
\overline{\Lambda}_\sigma: f\in X_\diamond\mapsto \overline{\Lambda}_\sigma(f)= \int_{0}^{1}\Lambda_\sigma  \left( \alpha f\right) \text{d}\alpha\in X_\diamond'.
\end{equation*}
Our main result is the Monotonicity Principle for the operator $\overline{\Lambda} _{\sigma }$. To be more precise, we prove
\[
\sigma_1\leq\sigma_2\quad\Longrightarrow\quad\overline\Lambda_{\sigma_1}  
\leq\overline{\Lambda}_{\sigma_2},
\]
where $\overline\Lambda_{\sigma_1}  
\leq\overline{\Lambda}_{\sigma_2}$ means  $\left\langle\overline{\Lambda}_{\sigma_1}  \left( f\right) ,f \right\rangle
\leq \left\langle\overline{\Lambda}_{\sigma_2}  \left( f\right),f\right\rangle$ for any $f\in X_\diamond$.

The key factor in achieving this result is Theorem \ref{transferthm}, which deals with the transfer of Dirichlet Energy to the power product for the Average DtN operator $\overline{\Lambda}_\sigma$. More precisely, we prove
\begin{equation}\label{detrnf}
(\mathbb F_\sigma\circ\mathbb{U}_\sigma) (f)=\left\langle\overline{\Lambda}_\sigma  \left( f\right) ,f \right\rangle
\quad\forall f\in X_\diamond,
\end{equation}
where operator $\mathbb{U}_\sigma$ maps the boundary data $f$ to the corresponding solution $u^f$ of (\ref{gproblem1}), i.e.
\begin{equation*}
\mathbb{U}_\sigma:f\in X_\diamond\to u^f\in W^{1,p}(\Omega).
\end{equation*}
The proof of \eqref{detrnf} is based on a fundamental result obtained in Proposition \ref{gateauxprop}. Specifically, we prove that the G\^{a}teaux derivative operator of $\mathbb{F}_\sigma\circ\mathbb{U}_\sigma$, with respect to the boundary data $f$, is equal to the DtN operator $\Lambda_\sigma$, i.e.
\begin{equation}\label{fund_rel_der}
    \text{d}(\mathbb F_\sigma\circ\mathbb{U}_\sigma) =\Lambda_\sigma.
\end{equation}

To identify the key differences between our problem and the  $p$-Laplacian / linear ($p=2$) cases, it is worth noting that \eqref{detrnf} is replaced by
\begin{equation*}
(\mathbb F_\sigma\circ\mathbb{U}_\sigma) (f)=p^{-1} \left\langle{\Lambda}_\sigma  \left( f\right) ,f \right\rangle
\quad\forall f\in X_\diamond,
\end{equation*}
whereas \eqref{fund_rel_der} remains unchanged.

The paper is organized as follows: in Section \ref{FoP}, we describe the problem together with the preliminaries required for its analysis; in Section \ref{Gderivative} we study the behaviour of the solution of Problem (\ref{gproblem1}) and of the Dirichlet Energy with respect a variation of the boundary data; in Section \ref{MP}, we prove the main result and, eventually, in Section \ref{Concl}, we provide some important conclusions. 

\bigskip

\section{Foundations of the problem}
\label{FoP}
Throughout this paper, $\Omega$ is the region occupied by the conducting material. We assume $\Omega\subset\R^n$, $n\geq 2$, to be an open bounded domain with Lipschitz boundary.
We denote by $\mathbf{\hat{n}}$ the outer unit normal defined on $\partial\Omega$, by $\langle \cdot, \cdot\rangle$ the integral scalar product on $\partial \Omega$ and by $V$ and $S$ the $n$-dimensional and the $(n-1)$-dimensional Hausdorff measure, respectively. Moreover, we denote
\[
L^\infty_+(\Omega):=\{\theta\in L^\infty(\Omega)\ |\ \theta\geq c_0\ \text{a.e. in}\ \Omega, \ \text{for some positive constant}\ c_0\}.
\]
Then, for $1<p<+\infty$, $W^{1,p}_0(\Omega)$ is the closure set of $C_0^1(\Omega)$ with respect to the $W^{1,p}$-norm. 

Furthermore, the applied boundary voltage $f$ belongs to the abstract trace space $$X=W^{1,p}(\Omega)/W^{1,p}_0(\Omega)\approx B_{p}^{1-\frac 1p,p}(\partial\Omega),$$ that is also a Besov space (refer to \cite{JERISON1995161}, \cite[Chap. 17]{leoni17}, \cite[App.]{doi:10.1137/110838224}). In the sequel, by a little abuse of notation, we write $f\in X$ meaning that $f$ is a representative of an $X$-equivalence class, i.e. $f\in[a]_X$ where $[a]_X\in X$. We indicate that $X_\diamond$ is the set of elements in $X$ with zero average on $\partial\Omega$ with respect to the measure $S$.

\subsection{The physical problem}
Let $\sigma$ be a function representing the nonlinear electrical conductivity, i.e. ${\bf J}  (x)=  \sigma(x,E(x)){\bf E}(x)$ where ${\bf E}$ and ${\bf J}$ are the electric field and the electrical current density, respectively. In addition, let $E$ and $J$ be the magnitude of ${\bf E}$ and ${\bf J}$, respectively.

Stationary currents are governed by:
\begin{eqnarray}
\label{first_max}  \curl {\bf E}(x)& = &0\ \text{in}\ \Omega, \quad\bigintssss_{\bar x}^{\bar y} {\bf E}\cdot \mathbf{\hat{t}}\ \text{d} \ell=f(\bar x)-f(\bar y) \ \ \forall\bar x,\bar y\in\partial\Omega;\\
\label{fourth_max}  \dive{\bf J} (x)&=& 0\ \text{in}\ \Omega,\\
\label{gOhm}  {\bf J} (x)& = & \sigma(x,E(x)) \ {\bf E}(x)\ \text{in}\ \Omega,
\end{eqnarray}
where $f\in X_\diamond$ is the applied boundary voltage. Equations \eqref{first_max} and \eqref{fourth_max} come from Maxwell equations for stationary models.
We stress that equations \eqref{first_max} and \eqref{fourth_max} have to be meant in weak sense and 
 \begin{equation*}
\begin{split}
{\bf E}\in H_{\curl}(\Omega)&=\{ {\bf w} \in L^2(\Omega;\R^3) \ |\ \curl ({\bf w}) \in L^2(\Omega)\},\\
{\bf J}\in H_{\dive}(\Omega)&=\{ {\bf w} \in L^2(\Omega;\R^3) \ |\ \dive ({\bf w}) \in L^2(\Omega)\}.
\end{split}\end{equation*}
 The curvilinear integral appearing in \eqref{first_max} is well defined for ${\bf E}$ in $H_{\curl}(\Omega)$ and holds for any $C^1$-curve in $\Omega$ with extrema $\bar x$ and $\bar y\in\partial\Omega$ (see \cite{bossavit1998computational}). 
Though problem (\ref{first_max})-(\ref{gOhm}) is defined in $\R^3$, it is important to observe that all the forthcoming results hold in any dimension $n\geq 1$, once the scalar potential $u$ has been introduced.

\subsection{The Electrical Conductivity}
The well-posedness of the forward problem in Hadamard sense (see Section \ref{wpfp} below) is the minimal requirement to formulate the associated inverse problem. This objective is guaranteed by the following assumptions on $\sigma:\overline\Omega\times[0,+\infty[\to\R$:
\begin{enumerate}
\item[{\bf (H1)}] $x\in\overline\Omega\mapsto \sigma(x, E)$ is measurable $\forall E\ge0$; 
\item[{\bf (H2)}]
$E\in\ ]0,+\infty[\mapsto \sigma(x, E) E$ is strictly increasing for a.e. $x \in\overline\Omega$; 
\item[{\bf (H3)}]
$E\in\ [0,+\infty[\mapsto \sigma(x, E)$ is in $C([0,+\infty[)$ for a.e. $x \in\overline\Omega$;\\
\end{enumerate}
\vspace{-0.5cm}
\ Moreover, for a fixed $p\ge 2$, we have that:
\begin{enumerate}
\item[{\bf (H4)}] there exist three positive constants $\overline{\sigma}_2\geq \underline{\sigma}_1$ and $E_0>0$ 
such that: 
\[
\qquad \underline{\sigma}_1 \left(\frac{E}{E_0}\right)^{p-2}\leq\sigma(x, E)\leq \overline{\sigma}_2\max\left\{1,\left( \frac{E}{E_0} \right)^{p-2} \right\}\quad \text{for a.e.} \  x\in\overline\Omega\ \text{and}\ \forall E>0;
\]
\item[{\bf (H5)}]
there exists $c>0$ such that:
\[(\sigma(x,E_2){\bf E}_2-\sigma(x,E_1){\bf E}_1)\cdot( {\bf E}_2-{\bf E}_1)\geq c|{\bf E}_2-{\bf E}_1|^p\quad\text{for a.e.}\ x\in\overline\Omega\ \text{and}\ \forall\ {\bf E}_1,{\bf E}_2\in\R^n
.\]
\end{enumerate}

Parameter $p$ appearing in (H4) and (H5) is related to $W^{1,p}(\Omega)$, the functional space to which the scalar potential belongs. Assumption (H5) is a generalization of a known vector inequality \cite[Sec.12, eq (I)]{lindqvist2017notes}. Moreover, it is also used in \cite[eq. (3.4)]{lam2020consistency} for the particular case of $p=2$. We remark that (H3) does not ask for the continuous dependence of conductivity upon the space variable $x$. In other terms, one can treat materials with abrupt discontinuity of the electrical conductivity.

We stress that assumptions (H1)-(H5) are quite general (see also figure \ref{figure1}). Indeed, our objective is to generalize the $p$-Laplacian and linear cases (see \eqref{pOhm_potential} and \eqref{Ohm_potential} below). For instance, polynomial electrical conductivities
\begin{equation}
    \label{poly_sigma}
\sigma(x, E)=\sum_{k=0}^N\theta _{k}\left( {x}\right) {E}^{k}\quad \text{for a.e.} \ x\in\overline\Omega\ \text{and}\ \forall E>0
\end{equation}
obviously satisfy  (H1)-(H3). Moreover, polynomial \eqref{poly_sigma} satisfies (H4) with $p=N+2$ and $\underline\sigma_1=c_N$, where $c_N>0$ is the essential infimum of $\theta_N$ as in the definition of $L^\infty_+(\Omega)$, and $\overline\sigma_2$ is the maximum among the essential supremums of $\theta_k$, $k=0,...,N$. Hypothesis (H5) is nothing but a generalization of the standard inequality (see \cite[Sec.12, eq (I)]{lindqvist2017notes})
\begin{equation}
    \label{trivial_p}
(E_2^{k}\ {\bf E}_2-E_1^{k}\ {\bf E}_1) \cdot ({\bf E}_2-{\bf E}_1)\geq \frac 1{2^{k+1}}|{\bf E}_2-{\bf E}_1|^{k+2} \quad \forall k\geq 0.
\end{equation}
Indeed, by multiplying 
(\ref{trivial_p}) with $\theta_k(x)$ and by summing from $0$ to $N$, we have
\[
\left(\sigma(x, E_2)\ {\bf E}_2-\sigma(x, E_1){\bf E}_1\right) \cdot ({\bf E}_2-{\bf E}_1)\geq \dfrac {\theta_N(x)}{2^{N+1}}|{\bf E}_2-{\bf E}_1|^{N+2}\geq\dfrac{ c_N}{2^{N+1}}|{\bf E}_2-{\bf E}_1|^{p}
,
\]
where $c_N>0$ is the essential infimum of $\theta_N$, as aformentioned. In other terms, \eqref{poly_sigma} satisfies (H5).

\begin{figure}[!ht]
	\centering
	\includegraphics[width=0.45\textwidth]{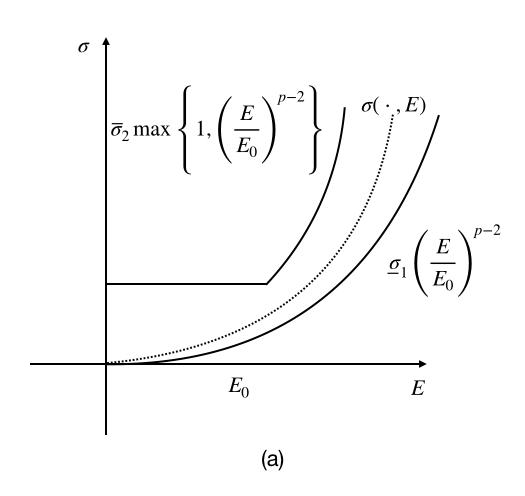}
	\includegraphics[width=0.45\textwidth]{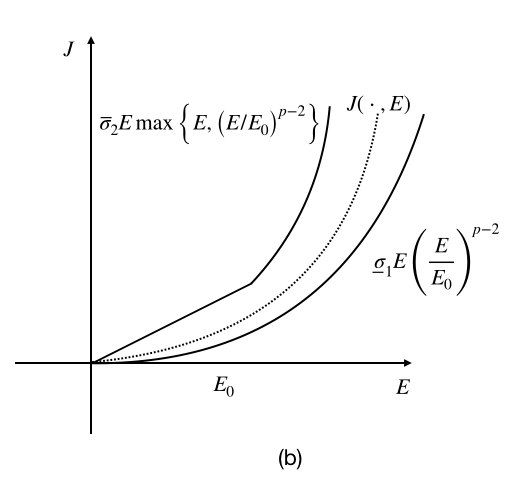}
	\caption{Impact of (H4) on the constitutive relationship in terms of electrical conductivity (a) and current density (b). Solid lines corresponds to the upper and lower constraints to either $\sigma$ or ${\bf J}$.}
	\label{figure1}
\end{figure}



In the literature, there are many practical  examples with constitutive relationships satisfying assumptions (H1)-(H5).

The monomial case, corresponding to the $p$-Laplacian, for some applicable situations is referred in \cite{doi:10.1137/110838224,brander2015enclosure,brander2016calderon,brander2018superconductive, gorb2012blow}. Polynomial constitutive relationships can be found in nonlinear inverse problems for incompressible hyperelastic materials \cite{ferreira2012uniqueness} and piezoelectric materials \cite{nakamura2009identification}. Other examples of nonlinear constitutive relationships for conductive materials complying hypotheses (H1)-(H5) are given in \cite{boucher2018interest,lupo1996field,donzel2011nonlinear}.



As a final remark, we mention some cases which cannot be treated in our framework. These cases are those of exponential constitutive relationships (see \cite{bueno2008sno2, zha2012prominent,corovic2013modeling}), the $p$-Laplacian for $1<p<2$ (see \cite{dibenedetto2012degenerate, yuan2005extinction} and reference therein) and fractional Laplacian \cite{bueno2014fractional}.

\subsection{Scalar potential and Dirichlet Energy functional}
In terms of the electrical scalar potential, that is ${\bf E}(x)=-\nabla u(x)$ with $u\in W^{1,p}(\Omega)$, the Ohm's law \eqref{gOhm} is
 \begin{equation}
 \label{gOhm_potential}
 {\bf J} (x)=- \sigma (x, |\nabla u(x)|)\nabla u(x).
 \end{equation}
Therefore, by \eqref{first_max}, \eqref{fourth_max} and \eqref{gOhm_potential}, the electrical potential $u$ 
solves the steady current problem:
 \begin{equation}\label{gproblem}
\begin{cases}
\dive\Big(\sigma (x, |\nabla u(x)|) \nabla u (x)\Big) =0\ \text{in }\Omega\vspace{0.2cm}\\
u(x) =f(x)\qquad\qquad\qquad\quad\  \text{on }\partial\Omega.
\end{cases}
\end{equation}
Here $u$ satisfies the boundary condition in the sense that $u-f\in W_0^{1,p}(\Omega)$ and we write $u|_{\partial\Omega}=f$. Specifically, problem \eqref{gproblem} is meant in the weak sense, that is
\begin{equation*}
\int_{\Omega }\sigma \left( x,| \nabla u(x) |\right) \nabla u (x) \cdot\nabla \varphi (x)\ \text{d}x=0\quad\forall\varphi\in W^{1,p}_0(\Omega).
\end{equation*}
The solution $u$ of \eqref{gproblem} is variationally characterized as
\begin{equation}\label{gminimum}
\argmin\left\{ \mathbb{F}_\sigma\left( u\right)\ :\ u\in W^{1,p}(\Omega), \ u|_{\partial\Omega}=f\right\}. 
\end{equation}
In (\ref{gminimum}), functional $\mathbb{F}_\sigma\left( u\right)$ is the Dirichlet Energy
\begin{equation}
\label{genergy}
\mathbb{F}_\sigma
\left(  u \right) = \int_\Omega Q_\sigma (x,|\nabla u(x)|)\ \text{d}x
\end{equation} 
and $Q_\sigma$ is the Dirichlet Energy density
\begin{equation}
\label{gQ}
Q_{\sigma}
\left( x,E\right)  =\int_{0}^{E} \sigma\left( x,\xi \right)\xi  \text{d}\xi\quad \text{for a.e.}\ x\in\overline\Omega\ \text{and}\ \forall E\geq0.
\end{equation}

\subsection{Well-posedness of the Forward Problem} \label{wpfp}

To obtain the existence of the solution of the forward problem (\ref{gminimum}), we use (H1), the left inequality in (H4) and a simpler version of (H2) with $\sigma(x,E)E$ being weakly  increasing. The proof follows from standard direct methods of calculus of variations (see e.g. \cite[Sec.s 4,5]{giusti2003direct}, \cite{dacorogna2014introduction}). Indeed, the convexity and coercivity of $Q_\sigma$ with respect to $E$ are the key factors. 

Uniqueness of the solution follows from (H2) that corresponds to the strict convexity of $Q_\sigma$ with respect to $E$.

The continuity of the solution of the forward problem with respect to the boundary data $f$ in a suitable norm follows from Lemma \ref{lpconv2}.

Eventually, we stress that problems, as in \eqref{gminimum}, have been broadly studied in various fields of mathematics (see e.g. \cite{della2019second,della2017sharp,piscitelli2016anonlocal} and reference therein).

\subsection{Connection among $\sigma$, $J_\sigma$ and $Q_\sigma$}

\begin{figure}[!ht]
	\centering
	\includegraphics[width=0.45\textwidth]{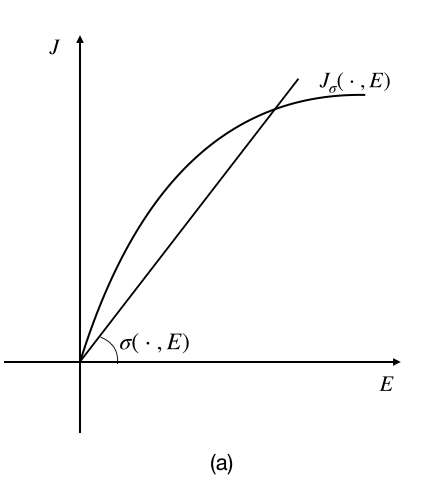}
	\includegraphics[width=0.45\textwidth]{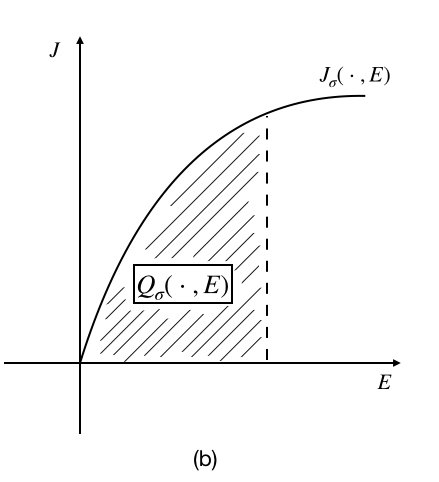}
	\caption{The nonlinear case. 
	For any given spatial point in the region $\Omega$, (a) the electrical conductivity $\sigma(\cdot,E)$ is the secant line to the graph of the function $J_\sigma(\cdot,E)$;
	(b) $Q_\sigma (\cdot, E)$ is the area of the sub-graph of $J_\sigma(\cdot, E)$.
	}
	\label{J(E)}
\end{figure}

The Ohm's law \eqref{gOhm} is isotropic and local, i.e. $\bf{J}$ is parallel to $\bf{E}$ and $J_\sigma$ depends on the position $x$ and on the magnitude of the electric field $E$ at the same location $x$.

By \eqref{gOhm_potential}, the Dirichlet Energy density can be also written as 
\begin{equation}\label{density}
Q_{\sigma}
\left( x,E \right)  =\int_{0}^{E}{J_\sigma} ({ x, \xi}) \ \text{d}{ \xi}\quad\text{for a.e.} \ x\in\overline\Omega\ \text{and}\ \forall E> 0.
\end{equation}
Relation \eqref{gOhm} gives the electrical current density as
\[
J_\sigma (x, E)=\sigma(x, E)E\quad \text{for a.e.} \ x \in\overline\Omega\ \text{and}\ \forall E>0.
\]
Moreover, relation \eqref{gQ} gives the electrical conductivity as
\[
\sigma\left( x,E\right) =E^{-1}J_\sigma (x, E)=E^{-1}\partial_E Q_\sigma
\left( x,E\right)\quad \text{for a.e.} \ x \in\overline\Omega\ \text{and}\ \forall E> 0.
\]
The electrical conductivity $\sigma(x,E)$ is the secant line to the graph of the function $J_\sigma(x,E(x))$ and $Q_\sigma (x, E(x))$ is the area of the sub-graph of $J_\sigma(x, E(x))$. For a geometric interpretation, see Figure \ref{J(E)}.

In the special case of $\sigma(x,E(x))=\theta(x)E(x)^{p-2}$, $1<p<+\infty$, the electrical current density is given by
\begin{equation}  \label{pOhm_potential}{\bf J}(x)=-\theta(x)| \nabla u(x)|^{p-2}\nabla u(x).
  \end{equation} 
This leads to the study of a steady current problem involving the $p$-Laplacian.
 \begin{figure}[!ht]
	\centering
	\includegraphics[width=0.45\textwidth]{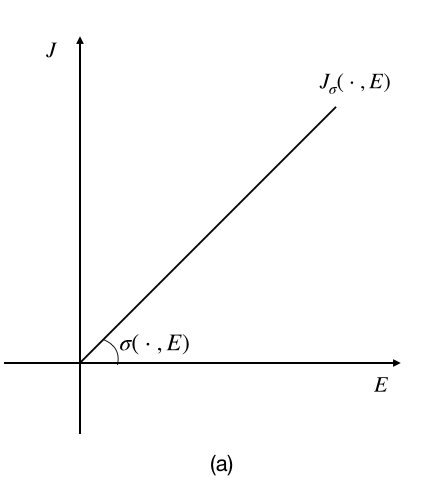}
	\includegraphics[width=0.45\textwidth]{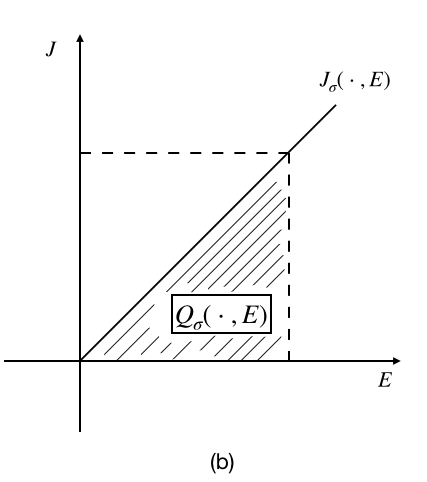}
	\caption{The linear case. For any given spatial point of the region $\Omega$, (a) the electrical conductivity is both the secant and the tangent line to the graph of the function $J_\sigma(\cdot, E)$; (b) the area of the subgraph, that is $Q_\sigma (\cdot, E)$, and the area of the super-graph are both equal to a half of the ohmic power density $J\! E$ absorbed by the system.
	}
	\label{linear}
\end{figure}
When ${\sigma}$ is independent of $E$, we have the standard linear model. More precisely, $\sigma(x,E(x))=\sigma(x)$ and consequently
 \begin{equation}  \label{Ohm_potential}{\bf J}(x)=-\sigma(x)\nabla u(x).
 \end{equation}
In this linear case, we do not need hypothesis (H3) and, since $d E^2=\frac 1 2 E \ \text{d}E$, the integral \eqref{density} is a half of the ohmic power absorbed by the system (refer to Figure \ref{linear}), related to the Joule effect:
\begin{equation*}
Q_\sigma\left( x,E (x)\right)  =\frac 12J_\sigma (x, E(x)) {E(x)}=\frac 12\sigma (x) {E(x)^2}\quad \text{for a.e.}\ x\in\overline\Omega\ \text{and}\ \forall E(x)>0.
\end{equation*}

\subsection{The DtN operator.}\label{theDtNoperator} One of the quite general ways to represent boundary measurements is by considering the Dirichlet-to-Neumann (DtN) operator
 \begin{equation*}
\Lambda_\sigma   :f\in X_\diamond\mapsto -{\bf J}^f\cdot \mathbf{\hat n}
\vert_{\partial \Omega}=\sigma\ 
\partial_nu^f|_{\partial\Omega} 
\in X_\diamond',
\end{equation*}
where $X_\diamond'$ is the dual space of $X$, ${\bf J}^f$, the current density produced by the boundary data $f$ and $u^f$, the minimizer of \eqref{gminimum}. We stress that $-{\bf J}^f\cdot \mathbf{\hat n}\vert_{\partial \Omega}$ represents the following linear functional
\begin{equation*}
\varphi \in X_\diamond \mapsto -\int_{\partial \Omega }\varphi (x)\ {\bf J}^f(x)\cdot\mathbf{\hat n}(x)\
\text{d}S.
\end{equation*}
Summing up, the DtN operator evaluated at $f$ is equal to:
\begin{equation}
\label{w-DtN}
\langle \Lambda_\sigma  \left( f\right) ,\varphi\rangle
=-\int_{\partial \Omega }\varphi (x)\ {\bf J}^f(x)\cdot\mathbf{\hat n}(x)\
\text{d}S\quad \forall\varphi \in X_\diamond.
\end{equation}

The minus sign in the definition is because we consider passive conducting material. Specifically, $-{\bf J}^f\cdot \mathbf{\hat n}$ corresponds to the current density entering the conductor through $\partial \Omega$. 
It is worth noting that the injectivity of the DtN operator is guaranteed by the assumption of zero average of $f$. Indeed, ${\bf J}^f$ is invariant up to an additive constant in $f$.

Furthermore, by testing the DtN operator \eqref{w-DtN} with the minimizer $u^f$ of \eqref{gminimum} and using a divergence Theorem, we obtain the ohmic power dissipated by the conducting material:
\begin{equation}
\label{g_formDtN}
\langle \Lambda_\sigma  \left( f\right) ,f\rangle
=\int_{\Omega } {J_\sigma}^f( x ,E(x)){ E}(x)\ \text{d}x.
\end{equation}
If $\varphi\neq f$ in \eqref{w-DtN}, we have the so-called
\emph{virtual power} product that plays an important role since it is equal to the G\^{a}teaux derivative of the Dirichlet Energy $\mathbb{F}_\sigma$ when evaluated at the solution $u^f$ (see Section \ref{Gderivative}).

When the nonlinear constitutive relation \eqref{gOhm_potential} holds, the DtN operator is
\begin{equation*}
\Lambda_\sigma  :f\in X_\diamond
\mapsto \sigma\left( x, \left\vert \nabla u^f\right\vert\right) \partial_n u^f \in X_\diamond'.
\end{equation*}%
In weak form, the DtN operator is
\begin{equation*}
\langle \Lambda_\sigma  \left( f\right) ,\varphi\rangle
=\int_{\partial \Omega }\varphi (x) {\sigma}\left( x, \left\vert \nabla u^f(x)\right\vert\right)  \partial_{{{n}}} u^f(x)\text{d}S\quad\forall \varphi\in X_\diamond.
\end{equation*}
 In the $p$-Laplacian \eqref{pOhm_potential} and in the linear \eqref{Ohm_potential} case, the DtN operators occupy the forms
 \begin{equation*}
 \langle \Lambda_\sigma  \left( f\right) ,\varphi\rangle=
 \int_{\partial \Omega }\varphi (x)\theta\left( x\right)  |\nabla u^f (x)|^{p-2} {\partial_{{{n}}} u(x)}\ \text{d}S\quad\forall \varphi\in X_\diamond
 \end{equation*}
 and 
 \begin{equation*}
 \langle \Lambda_\sigma  \left( f\right) ,\varphi\rangle=
 \int_{\partial \Omega }\varphi (x){\sigma}\left( x\right)  {\partial_{{{n}}} u^f(x)}\ \text{d}S\quad\forall \varphi\in X_\diamond,
 \end{equation*}
 respectively.

Moreover, the power product \eqref{g_formDtN} is proportional to the absorbed ohmic power, represented by the area of the dashed rectangle in figure \ref{linear}.

\section{The G\^{a}teaux derivative of the Dirichlet Energy}\label{Gderivative}
A key role in the problem, we are dealing with, is played by the G\^{a}teaux derivative of the Dirichlet Energy $\mathbb F_\sigma$ and by the G\^{a}teaux derivative of the composite function $\mathbb F_\sigma\circ \mathbb U_\sigma$, where
\begin{equation}
\label{U}
\mathbb{U}_\sigma:f\in X_\diamond\to u^f\in W^{1,p}(\Omega).
\end{equation}
Operator $\mathbb{U}_\sigma$ maps the boundary data $f$ to the solution $u^f$ of problem \eqref{gminimum}.

It is worth noting that the results of this Section are the foundations of the Monotonicity Principle in Section \ref{MP}.
Firstly, we prove a convergence result (Lemma \ref{lpconv2}), regarding $\nabla u$, i.e. the Electric field ${\bf E}$, with respect to the boundary data. Then, in Proposition \ref{pertF}, we study the G\^{a}teaux derivative of the Dirichlet Energy. As corollary, we prove that when evaluated on a physical solution, this G\^{a}teaux derivative is equal to the virtual power product (Corollary \ref{corPertF}).
Eventually, in Proposition \ref{gateauxprop}, we prove that the G\^{a}teaux derivative of the composite operator $\mathbb{F}_\sigma \circ \mathbb{U}_\sigma$ is equal to the DtN operator $\Lambda_\sigma$. This is the key result for proving the Monotonicity Principle in Section \ref{MP}.

For notation and properties of G\^{a}teaux-derivative, we refer to \cite[Chap. 4]{10.5555/59523}). Considering the definition of $X_\diamond$, we stress that many terms appearing in this section depend only on the restriction of $f$ on the boundary of $\Omega$.

\subsection{A convergence result} Firstly, we show the following useful convergence result. 
\begin{lem}
\label{lpconv2}
Let $\Omega$ be an open bounded domain with Lipschitz boundary and $f,\varphi\in X_\diamond$. Then
\[
 \nabla u^{f+\varepsilon \varphi}\to\nabla u^f\ \text{in} \ {L^p(\Omega)}\quad\text{as}\ {\varepsilon\to 0},
\]
where $u^{f+\varepsilon \varphi}\in W^{1,p}(\Omega)$ is the minimizer of \eqref{gminimum} corresponding to the  boundary data $f+\varepsilon \varphi$.
\end{lem}
\begin{proof}For a fixed $|\varepsilon|<1$, using the divergence Theorem, we have\footnote{$^2$We  remind that $f$ and $\varphi$ are representatives of their equivalence classes in $X_\diamond$.}{$^2$}
\begin{equation}\label{Idiff}
\begin{split}
I:=\int_\Omega  {\sigma}\left( x, \left\vert \nabla u^{f+\varepsilon \varphi}(x)\right\vert\right)  {\nabla u^{f+\varepsilon \varphi}(x)}\cdot(  \nabla u^{f+\varepsilon \varphi}(x)- \nabla u^{f}(x))\text{d}x=\\=\varepsilon
\int_{\partial \Omega } \varphi(x) {\sigma}\left( x, \left\vert \nabla u^{f+\varepsilon \varphi}(x)\right\vert\right)  \partial_n u^{f+\varepsilon \varphi}(x)\text{d}S,
\end{split}
\end{equation}
\begin{equation}\label{IIdiff}
\begin{split}
II:=\int_\Omega  {\sigma}\left( x, \left\vert \nabla u^{f}(x)\right\vert\right)  {\nabla u^{f}(x)}\cdot(  \nabla u^{f+\varepsilon \varphi}(x)- \nabla u^{f}(x))\text{d}x=\\
=\varepsilon
\int_{\partial \Omega} \varphi(x) {\sigma}\left( x, \left\vert \nabla u^{f}(x)\right\vert\right)  \partial_n u^{f}(x)\text{d}S.
\end{split}
\end{equation}
By subtracting \eqref{IIdiff} from \eqref{Idiff} and using assumption (H5), we get a positive constant $c>0$ such that
\begin{equation}\label{lowconv}
I-II\geq c ||\nabla u^{f+\varepsilon \varphi}-\nabla u^f||^p_{p}.
\end{equation}
On the other hand, from (H4), we know that $\sigma(x,E)\leq \overline\sigma_2\max\left\{1,\left(\frac{E}{E_0}\right)^{p-2}\right\}$ and, therefore, by H\"older inequality, we have
\begin{equation*}
\begin{split}
I-II & =\varepsilon\int_{\partial \Omega } \varphi(x) \left[{\sigma}\left( x, \left\vert \nabla u^{f+\varepsilon \varphi}(x)\right\vert\right)  \partial_n u^{f+\varepsilon \varphi}(x)- {\sigma}\left( x, \left\vert \nabla u^{f}(x)\right\vert\right)  \partial_n u^{f}(x)\right]\text{d}S\\
& = \varepsilon\int_\Omega\nabla \varphi(x)\cdot\left[  {\sigma}\left( x, \left\vert \nabla u^{f+\varepsilon \varphi}(x)\right\vert\right)  {\nabla u^{f+\varepsilon \varphi}(x)}- {\sigma}\left( x, \left\vert \nabla u^{f}(x)\right\vert\right)  {\nabla u^{f}(x)}\right]\ \text{d}x\\
& \leq\varepsilon \overline\sigma_2\int_\Omega| \nabla \varphi(x)| \left[ \max\left\{\left\vert \nabla u^{f+\varepsilon \varphi}(x)\right\vert ,\frac{\left\vert \nabla u^{f+\varepsilon \varphi}(x)\right\vert^{p-1}}{E_0^{p-2}}\right\}\right. +\\
&\qquad\qquad\qquad\qquad\qquad\qquad\qquad\left.\max\left\{ \left\vert \nabla u^{f}(x)\right|,\frac{\left\vert \nabla u^{f}(x)\right\vert^{p-1}}{E_0^{p-2}} \right\} \right]\ \text{d}x.
\end{split}
\end{equation*}
Moreover, we have
\begin{equation}\label{uppconv}
\begin{split}
I-II
&\leq \varepsilon C\left[\max\{\left|\left|\nabla \varphi\right|\right|_2 \left\vert\left\vert \nabla u^{f+\varepsilon \varphi}\right\vert\right\vert_2, \left|\left|\nabla \varphi\right|\right|_p \left\vert\left\vert \nabla u^{f+\varepsilon \varphi}\right\vert\right\vert_p^{p-1}\}\right.\\
&\left.\qquad\qquad\qquad\qquad+\max\{\left|\left|\nabla \varphi\right|\right|_2 \left\vert\left\vert \nabla u^{f}\right\vert\right\vert_2, \left|\left|\nabla \varphi \right|\right|_p \left\vert\left\vert \nabla u^{f}\right\vert\right\vert_p^{p-1} \} \right]\\
&\leq \varepsilon C\left[\max\{\left|\left|\nabla \varphi\right|\right|_2 \left\vert\left\vert \nabla {f+\varepsilon\nabla \varphi}\right\vert\right\vert_2, \left|\left|\nabla \varphi\right|\right|_p \left\vert\left\vert \nabla {f+\varepsilon\nabla \varphi}\right\vert\right\vert_p^{p-1}\}\right.\\
&\left.\qquad\qquad\qquad\qquad+\max\{\left|\left|\nabla \varphi\right|\right|_2 \left\vert\left\vert \nabla {f}\right\vert\right\vert_2, \left|\left|\nabla \varphi \right|\right|_p \left\vert\left\vert \nabla {f}\right\vert\right\vert_p^{p-1} \} \right]\\
&\leq \varepsilon C\left[\max\left\{\left|\left|\nabla \varphi\right|\right|_2 (\left\vert\left\vert \nabla f\right|\right|_2+ \left\vert\left\vert\nabla \varphi\right\vert\right\vert_2), \left|\left|\nabla \varphi\right|\right|_p( \left\vert\left\vert \nabla f\right|\right|_p^{p-1}+ \left\vert\left\vert\nabla \varphi\right\vert\right\vert_p^{p-1}) \right\}\right.\\
&\left.\qquad\qquad\qquad\qquad+\max\{\left|\left|\nabla \varphi\right|\right|_2 \left\vert\left\vert \nabla {f}\right\vert\right\vert_2, \left|\left|\nabla \varphi \right|\right|_p \left\vert\left\vert \nabla {f}\right\vert\right\vert_p^{p-1} \} \right].\\
\end{split}
\end{equation}
The quantity in the squared bracket is finite and hence, by \eqref{lowconv} and \eqref{uppconv}, 
we get
\begin{equation*}
||\nabla u^{f+\varepsilon \varphi}-\nabla u^f||^p_{p}\leq C\varepsilon,
\end{equation*}
where $C$ is a positive constant independent of $\varepsilon$.
Therefore, the conclusion follows by passing the limit $\varepsilon\to 0$.
\end{proof}

\subsection{The perturbation of the Dirichlet Energy with respect to the minimizer}
Before proving the main result, we observe that (H2) gives the convexity of $Q_\sigma(x,E)$ with respect to $E$ and, since $Q_\sigma(x,0)=0$, then $\partial_E Q_\sigma(x,E)$ is increasing with respect to $E>0$ for a.e. $x\in\overline\Omega$. This monotonic behaviour of $\partial_E Q_\sigma(x,E)$ leads to the following inequalities for a.e. $x\in\overline\Omega$:
\begin{equation}
    \label{pconv}
    \begin{split}
0\leq \partial_E Q_\sigma (x, E_1)(E_2-E_1)& \leq Q_\sigma(x, E_2)-Q_\sigma(x, E_1)\\&\leq \partial_E Q_\sigma (x, E_2)(E_2-E_1)\quad\text{for any}\ 0<E_1\leq E_2.
    \end{split}
\end{equation}
Now, we study the G\^{a}teaux-derivative of \eqref{genergy}. 
\begin{prop}\label{pertF} Let $\Omega$ be a bounded connected domain with Lipschitz boundary and $u,\varphi\in W^{1,p}(\Omega)$. 
Then
\begin{align}
\label{dFlib}
\text{\emph d}\mathbb F_\sigma (u;\varphi)&=\mathbb F_\sigma'(u)\varphi=
\int_\Omega  {\sigma}\left( x, \left\vert \nabla u(x)\right\vert\right)  \nabla u(x)\cdot \nabla \varphi (x)\ \text{\emph d}x.
\end{align}
\end{prop}
\begin{proof} 
For any $\varepsilon\in\R$, we have 
\begin{equation*}
\begin{split}
\mathbb F_\sigma(u)&=\int_\Omega Q_\sigma(x, |\nabla u(x)|)\text{d}x,\\
\mathbb F_\sigma(u+\varepsilon\varphi)&= \int_\Omega Q_\sigma\left(x, \left|\nabla {u}(x)+\varepsilon\nabla \varphi(x)\right|\right)\text{d}x.
\end{split}
\end{equation*}
Let us study the incremental ratio:
\begin{equation}
    \label{rapp_incr}
\frac{\mathbb F_\sigma(u+\varepsilon\varphi)-\mathbb F_\sigma (u)}\varepsilon= \frac 1{ \varepsilon }\int_\Omega 
Q_\sigma\left(x,|\nabla {u}(x)+\varepsilon \nabla \varphi(x)|\right)-Q_\sigma\left(x, |\nabla u(x)|
\right)\text{d}x.
\end{equation}
The magnitude of the integrand function can be bounded easily from above. Indeed, by \eqref{pconv}, we have
\begin{equation}
\label{upp}
\begin{split}
& \left| \frac {
Q_\sigma\left(x,|\nabla {u}(x)+\varepsilon \nabla \varphi(x)|\right)-Q_\sigma\left(x, |\nabla u(x)|
\right)}\varepsilon\right|\\
&\qquad\leq\frac 1{| \varepsilon|} \sigma \left(x,\left| \nabla {u}(x)|+|\varepsilon \nabla \varphi(x)\right|\right)(\left|\nabla u(x)|+|\varepsilon \nabla \varphi(x)\right|)|\varepsilon\nabla \varphi(x)| \\
&\qquad\leq \sigma \left(x,\left| \nabla u(x)|+| \nabla \varphi(x)\right|\right)(\left|\nabla u(x)|+| \nabla \varphi(x)\right|)|\nabla \varphi(x)|, 
\end{split}
\end{equation} 
for any $|\varepsilon| <1$.
By assumption (H4), the last term in \eqref{upp} is a $L^1$ function and hence, by Lebesgue Dominate Convergence Theorem, we can pass to the limit in \eqref{rapp_incr}, as $\varepsilon\to 0$. Then, by assumption (H2), we have that $Q_\sigma(x, \cdot)$ is in $C^1([0,+\infty[)$ for a.e $x\in\overline\Omega$ and hence
\begin{equation}\label{rapp_incr_int}
\begin{split}
&\frac{Q_\sigma\left(x,|\nabla u(x)+\varepsilon \nabla \varphi(x)|\right)-Q_\sigma\left(x, |\nabla u(x)|\right)}\varepsilon\\
&\qquad\qquad=\sigma \left(x,\xi_\varepsilon \right)\xi_\varepsilon \frac{ \left|\nabla u(x)+\varepsilon \nabla \varphi(x)\right|-\left|\nabla u(x)\right|}\varepsilon\\
\end{split}
\end{equation}
with
\begin{equation}
\label{interval}
\xi_\varepsilon\in [\min\{\left|\nabla u(x)\right|, \left|\nabla u+\varepsilon \nabla \varphi(x)\right|\},\max\{\left|\nabla u(x)\right|,\left|\nabla u+\varepsilon \nabla \varphi(x)\right|\}],
\end{equation}
where the dependence of $\xi_\varepsilon$ upon $x$ is apparent.

The last term in \eqref{rapp_incr_int} is equal to
\begin{equation}\label{lim_e}
\begin{sistema}
\sigma \left(x,\xi_\varepsilon \right)\xi_\varepsilon \dfrac{2\nabla u(x)\cdot\nabla \varphi (x)+\varepsilon \left|\nabla\varphi(x)\right|^2}{|\nabla u(x)|+|\nabla u(x)+\varepsilon \nabla \varphi(x)||},\quad\text{if}\ |\nabla u(x)|\neq 0,\vspace{0.2cm}\\
\sigma \left(x,\xi_\varepsilon \right)\xi_\varepsilon \sign(\varepsilon)|\nabla\varphi (x)|,\qquad\qquad\qquad\quad\ \text{otherwise}.
\end{sistema}
\end{equation}
Hence, since \eqref{interval} holds, $\xi_\varepsilon\to |\nabla u(x)|$ as $\varepsilon\to 0$ for a. e $x\in\overline\Omega$. We pass to the limit as $\varepsilon\to 0$ in \eqref{lim_e}, \eqref{rapp_incr_int} and \eqref{rapp_incr} to obtain \eqref{dFlib}.
\end{proof}
Proposition \ref{pertF} implies the following corollary by replacing $u$ with $u^f$, the minimizer of \eqref{gminimum}.
\begin{cor} \label{corPertF}
Let $\Omega$ be a bounded connected domain with Lipschitz boundary, $f\in X_\diamond$ and $\varphi\in W^{1,p}(\Omega)$. Then
\begin{align}
\label{dF}
\text{\emph d}\mathbb F_\sigma (u^f;\varphi)&=\mathbb F_\sigma'(u^f)\varphi=
\langle\Lambda_\sigma (f),\varphi\rangle,
\end{align}
where $u^{f}$ is the minimizer of \eqref{gminimum} corresponding to the boundary data ${f}$. 
\end{cor}
 It is worth noting that also the G\^{a}teaux derivative $\text{d}\mathbb F_\sigma (u^f;\varphi)$ appearing in \eqref{dF} depends only on the restriction of $\varphi$ on the boundary of $\Omega$.

\subsection{The perturbation of $\mathbb F_\sigma$ with respect to the boundary values} In this section, we analyze the G\^{a}teaux derivative operator for the composition of the Dirichlet Energy functional $\mathbb F_\sigma$ and the operator $\mathbb U_\sigma$, defined in \eqref{U}. 
\begin{prop} \label{gateauxprop} Let $\Omega$ be a bounded connected domain with Lipschitz boundary and $f\in X_\diamond$. Then $\text{\emph d} (\mathbb F_\sigma\circ\mathbb{U_\sigma})=(\mathbb F_\sigma\circ \mathbb U_\sigma)'=\Lambda_\sigma$, i.e.
\begin{equation}
    \label{dG}
    \text{\emph d}(\mathbb F_\sigma\circ\mathbb{U_\sigma}) (f;\varphi)=(\mathbb F_\sigma\circ \mathbb U_\sigma)'(f)\varphi=\langle\Lambda_\sigma (f),\varphi\rangle \quad \forall \varphi\in X_\diamond.
\end{equation}
\end{prop}
\begin{proof} For sake of simplicity, we set
\begin{equation*}
\mathbb G_\sigma(f):=(\mathbb F_\sigma\circ\mathbb{U}_\sigma)(f)=\mathbb{F}_\sigma\left(  u^f \right).
\end{equation*}
For any $\varepsilon>0$, we consider
\begin{equation*}
\begin{split}
\mathbb G_\sigma(f)&=\int_\Omega Q_\sigma(x, |\nabla u^f(x)|)\text{d}x,\\
\mathbb G_\sigma(f+\varepsilon\varphi)&= \int_\Omega Q_\sigma\left(x, \left|\nabla u^{f+\varepsilon\varphi}(x)\right|\right)\text{d}x\leq \int_\Omega Q_\sigma\left(x, \left|\nabla u^{f}(x)+\varepsilon\nabla \varphi(x)\right|\right)\text{d}x. 
\end{split}
\end{equation*}
Let us consider the incremental ratio $\left[\mathbb G_\sigma(f+\varepsilon\varphi)-\mathbb G_\sigma (f)\right]/\varepsilon$. On one hand, we have
\begin{equation}
    \label{rapp_incrG}
\frac{\mathbb G_\sigma(f+\varepsilon\varphi)-\mathbb G_\sigma (f)}\varepsilon\leq \frac 1{ \varepsilon }\int_\Omega 
Q_\sigma\left(x,|\nabla u^{f}+\varepsilon \nabla \varphi(x)|\right)-Q_\sigma\left(x, |\nabla u^f(x)|
\right)\text{d}x.
\end{equation}
The second term in \eqref{rapp_incrG} is equal to the second term in \eqref{rapp_incr}, therefore we can pass to the limit and we have
\begin{equation}\label{limsupconv}
\begin{split}
&\limsup_{\varepsilon\to 0} \frac{\mathbb G_\sigma(f+\varepsilon\varphi)-\mathbb G_\sigma (f)}\varepsilon\\
&\leq\lim_{\varepsilon\to 0} \frac 1{ \varepsilon }\int_\Omega Q_\sigma\left(x,|\nabla u^{f}(x)+\varepsilon \nabla \varphi(x)|\right)-Q_\sigma\left(x, |\nabla u^f(x)|\right)\text{d}x\\
&=\int_\Omega  {\sigma}\left( x, \left\vert \nabla u^f(x)\right\vert\right)  \nabla u^{f}(x)\cdot \nabla \varphi (x)\ \text{\emph d}x.
\end{split}
\end{equation}
On the other hand, we have
\begin{equation}\label{lowlim}
\begin{split}
\frac{\mathbb G_\sigma(f+\varepsilon\varphi)-\mathbb G_\sigma(f)}\varepsilon & \geq  \frac 1{ \varepsilon }\int_\Omega Q_\sigma\left(x,|\nabla u^{f+\varepsilon \varphi}(x)|\right)-Q_\sigma\left(x, |\nabla u^{f+\varepsilon\varphi}(x)-\varepsilon \nabla\varphi(x)|\right)\text{d}x.
\end{split}
\end{equation}
In Lemma \ref{lpconv2}, we showed that $\nabla u^{f+\varepsilon\varphi}\to\nabla u^f$ strongly in $L^p(\Omega)$ when $\varepsilon\to 0$. Now, let us consider a sequence $\{\varepsilon_j\}_{j\in\N}$, such that $\varepsilon_j\to 0$ and 
\begin{equation}
    \label{limj}
\liminf_{\varepsilon\to 0^+} \frac{\mathbb G_\sigma(f+\varepsilon\varphi)-\mathbb G_\sigma(f)}\varepsilon=\lim_{j\to +\infty}  \frac{\mathbb G_\sigma(f+\varepsilon_j\varphi)-\mathbb G_\sigma(f)}{\varepsilon_j}.
\end{equation}
By standard arguments, we can say that there exists a subsequence $\{\varepsilon_{j_h}\}_{h\in\N}$ such that  $\nabla u^{f+\varepsilon_{j_h}\varphi}\to\nabla u^f$ a.e. in $\Omega$ and there exists a measurable real function $\psi$, defined on $\Omega$, such that
\[
|\nabla u^{f+\varepsilon_{j_h}\varphi}|\leq\psi,\quad\int_\Omega|\psi(x)|^p\ dx<+\infty.
\]
Now, we consider the integrand function in \eqref{lowlim}. By \eqref{pconv} and (H4), we have
\begin{equation*}
\begin{split}
&\left|\frac{Q_\sigma\left(x,|\nabla u^{f+\varepsilon_{j_h} \varphi}(x)|\right)-Q_\sigma\left(x, |\nabla u^{f+\varepsilon_{j_h}\varphi}(x)-\varepsilon_{j_h} \nabla\varphi(x)|\right)}{\varepsilon_{j_h}}
\right|\\
&\leq \sigma(x,|\nabla u^{f+\varepsilon_{j_h}\varphi}(x)|+|\varepsilon_{j_h}  \nabla\varphi(x)| )( |\nabla u^{f+\varepsilon_{j_h} \varphi}(x)|+|\varepsilon_{j_h} \nabla\varphi(x)|)|\nabla \varphi(x)|\\
& \leq C \max\{\tilde C,(|\nabla u^{f+\varepsilon_{j_h}\varphi}(x)|+|\varepsilon_{j_h}  \nabla\varphi(x)| )^{p-2}\}( |\nabla u^{f+\varepsilon_{j_h}\varphi}(x)|+|\varepsilon_{j_h} \nabla\varphi(x)|)|\nabla\varphi(x)|\\
& \leq C \max\{\tilde C,(\psi(x)+ | \nabla\varphi(x)| )^{p-2}\}( \psi(x)+ |\nabla\varphi(x)|)^2
\\
& \leq C \max\{\tilde C( ( \psi(x)+ |\nabla\varphi(x)|)^2), (\psi(x)+ | \nabla\varphi(x)|)^{p}\},
    \end{split}
\end{equation*}
for a.e. $x \in \Omega$ and for any $\varepsilon<1$. In the above expression $C$ and $\tilde C>0$ are appropriate positive constants independent of $\varepsilon_{j_h}$.

Now, recalling that $Q_\sigma(x, \cdot)$ is in $C^1([0,+\infty[)$ for a.e. $x\in\Omega$, the integrand function in \eqref{lowlim} can be written as 
\begin{equation}\label{rapp_incr_intG}
\begin{split}
&\frac{Q_\sigma\left(x,|\nabla u^{f+\varepsilon\varphi}(x)|\right)-Q_\sigma\left(x, |\nabla u^{f+\varepsilon \varphi}(x)-\varepsilon\nabla \varphi(x)|\right)}\varepsilon\\
&\qquad\qquad=\sigma \left(x,\xi_\varepsilon' \right)\xi_\varepsilon' \frac{ \left|\nabla u^{f+\varepsilon\varphi}(x)\right|-\left|\nabla u^{f+\varepsilon\varphi}(x)-\varepsilon\nabla \varphi(x)\right|}\varepsilon\\
\end{split}
\end{equation}
with
\begin{equation*}
\begin{split}
\xi_\varepsilon'\in &[\min\{\left|\nabla u^{f+\varepsilon \varphi}(x)\right|, \left|\nabla u^{f+\varepsilon \varphi}(x)-\varepsilon \nabla \varphi(x)\right|\},\\ 
&\qquad\qquad\qquad\qquad\qquad\qquad\max\{\left|\nabla u^{f+\varepsilon \varphi}(x)\right|,\left|\nabla u^{f+\varepsilon\varphi}(x)-\varepsilon \nabla \varphi(x)\right|\}],
\end{split}
\end{equation*}
where the dependence of $\xi_\varepsilon'$ upon $x$ is apparent.

The last term in \eqref{rapp_incr_intG} is equal to
\begin{equation}\label{lim_eG}
\begin{sistema}
\sigma \left(x,\xi_\varepsilon' \right)\xi_\varepsilon' \dfrac{2\nabla u^{f+\varepsilon\varphi}(x)\cdot\nabla \varphi (x)-\varepsilon \left|\nabla\varphi(x)\right|^2}{|\nabla u^{f+\varepsilon\varphi}(x)|+|\nabla u^{f+\varepsilon\varphi}(x)-\varepsilon \nabla \varphi(x)|},\quad\text{if}\ |\nabla u^f(x)|\neq 0,\vspace{0.2cm}\\
\sigma \left(x,\xi_\varepsilon' \right)\xi_\varepsilon' \sign(-\varepsilon)|\nabla\varphi (x)|,\qquad\qquad\qquad\qquad\qquad \text{otherwise}.
\end{sistema}
\end{equation}
In the first case, $\xi_\varepsilon'\to |\nabla u^f(x)|$ as $\varepsilon\to 0$ for a.e. $x\in\Omega$; in second the case, $|\nabla u^f(x)|=\lim_{\varepsilon\to 0}\xi_\varepsilon'=0$. We notice that, in both cases, the terms in \eqref{lim_eG} tend to ${\sigma}\left( x, \left\vert \nabla u^f(x)\right\vert\right)  \nabla u^f(x)\cdot \nabla \varphi (x)$, as $\varepsilon\to 0$. 

By applying Dominate Convergence Theorem, we can consider the limit as $\varepsilon_{j_h}\to 0$ in \eqref{lim_eG}, \eqref{rapp_incr_intG} and hence in the second term of \eqref{lowlim}. Hence, by \eqref{lowlim}, \eqref{limj}, \eqref{rapp_incr_intG} and \eqref{lim_eG}, we conclude:
\begin{equation}
\label{liminf_conv}
\begin{split}
&\liminf_{\varepsilon\to 0} \frac{\mathbb G_\sigma(f+\varepsilon\varphi)-\mathbb G_\sigma(f)}\varepsilon=\lim_{j\to +\infty}  \frac{\mathbb G_\sigma(f+\varepsilon_j\varphi)-\mathbb G_\sigma(f)}{\varepsilon_j}\\
&=\lim_{h\to +\infty}  \frac{\mathbb G_\sigma(f+\varepsilon_{j_h}\varphi)-\mathbb G_\sigma(f)}{\varepsilon_{j_h}}\\
&= \lim_{h\to\infty} \frac 1{ \varepsilon_{j_h} }\int_\Omega Q_\sigma\left(x,|\nabla u^{f+\varepsilon_{j_h} \varphi}(x)|\right)-Q_\sigma\left(x, |\nabla u^{f+\varepsilon_{j_h}\varphi}(x)-\varepsilon_{j_h} \nabla\varphi(x)|\right)\text{d}x\\
&= \int_\Omega  {\sigma}\left( x, |\nabla u^f(x)|\right) \nabla u^{f}(x)\cdot\nabla \varphi (x) \text{d}x.
\end{split}
\end{equation}
The conclusion follows by observing that \eqref{limsupconv} and \eqref{liminf_conv} imply \eqref{dG}.
\end{proof}

 \section{Monotonicity Principle}
 \label{MP}

In this section, we demonstrate a Monotonicity Principle for the Dirichlet Energy (Theorem \ref{monoten}), more precisely, 
\begin{equation}\label{m_genergy3}
\sigma_1\leq\sigma_2\quad\Longrightarrow\quad\mathbb F_{\sigma_1}(u_{1}^f) \leq \mathbb F_{\sigma_2} (u_{2}^f)\quad \forall f\in X_\diamond,
\end{equation}
where $u_{i}^f$ is the minimizer of \eqref{gminimum} with $\sigma=\sigma_i$, for $i=1,2$ and $f$ is the applied boundary voltage. We recall that, both in the $p$-Laplacian and in the linear cases, the Dirichlet Energy is proportional to the power product $\langle\Lambda_\sigma(f),f\rangle$ and, therefore, in these cases, Monotonicity Principle for the DtN operator easily follows from (\ref{m_genergy3}). 

In our nonlinear case \eqref{gOhm_potential}, the Dirichlet Energy is not proportional to the power product $\langle\Lambda_\sigma(f),f\rangle$. Rather, the power product $\langle\Lambda_\sigma(f),f\rangle$ is equal to the G\^{a}teaux derivative $\text{ d}\left(\mathbb  F_\sigma \circ \mathbb{U}_\sigma \right) (f;f)$ of the composite mapping $f\rightarrow\mathbb  F_\sigma\left(u^f\right)$ (see Proposition \ref{gateauxprop}). It is a fundamental difference between our analysis and previous ones. Then, starting from this latter equality, we relate the Dirichlet Energy to boundary data through the fundamental relation (see Theorem \ref{transferthm})
\begin{equation*}
\mathbb{F}_{\sigma}\left( u^f\right)=\left\langle\overline{\Lambda}_\sigma  \left( f\right) ,f \right\rangle
\quad\forall f\in X_\diamond.
\end{equation*}
The new operator $\overline{\Lambda}_\sigma$, we call Average DtN, is defined as
\begin{equation}
\label{average_flown}
\overline{\Lambda}_\sigma: f\in X_\diamond\mapsto  \int_{0}^{1}\Lambda_\sigma  \left( \alpha f\right) \text{d}\alpha\in X'_\diamond,
\end{equation}
where
\begin{equation*}
\langle\overline{\Lambda}_\sigma( f),\varphi\rangle=
 \int_{0}^{1}\left\langle\Lambda_\sigma  \left( \alpha f\right) , \varphi\right\rangle\text{d}\alpha
\quad\forall\varphi \in X_\diamond.
\end{equation*} 
Operator $\overline{\Lambda}_\sigma$ gives the average flown of the electrical current density through $\partial \Omega$ for an applied boundary potential of the type $\alpha f$, for $\alpha \in \left[0, 1\right]$. This is the key development for transferring the Monotonicity of the Dirichlet Energy to the boundary data and, in particular, the new operator $\overline\Lambda_\sigma$ replaces for nonlinear problems.

Eventually, in Theorem \ref{monothm}, we prove the Monotonicity Principle for operator $\overline{\Lambda}_\sigma$, i.e.
\[
\sigma_1\leq\sigma_2\quad\Longrightarrow\quad \left\langle\overline{\Lambda}_{\sigma_1}  \left( f\right) ,f \right\rangle
\leq \left\langle\overline{\Lambda}_{\sigma_2}  \left( f\right) ,f \right\rangle \quad\forall f\in X_\diamond,
\]
where we recall that $\sigma_1\leq\sigma_2$ means
\begin{equation}
    \label{defmon}
\sigma_1(x,E)\leq \sigma_2(x,E) \quad \text{for a.e.}\ x\in\overline\Omega\ \text{and}\ \forall\ E>0.
\end{equation}

We remark that Monotonicity Principle involves the knowledge of $\left\langle\overline{\Lambda}_\sigma \left( f\right) ,f \right\rangle=\int_0^1\langle \Lambda_\sigma (\alpha f),f\rangle\  \text{d}\alpha$.
From a physical standpoint, quantity $\langle \Lambda_\sigma (\alpha f),\alpha f\rangle$ is nothing but the electrical power $P(\alpha f)$ absorbed by the conductor when the boundary potential is $\alpha f$ (see Section \ref{theDtNoperator}). Therefore, key quantity  $\left\langle\overline{\Lambda}_\sigma \left( f\right) ,f \right\rangle$ is equal to $\int_0^1 \alpha^{-1} P(\alpha f) \text{d}\alpha$, that is, a weighted integral of the ohmic power dissipated in the conductor.

\subsection{Monotonicity Principle for the Dirichlet Energy} 
Firstly, we state Monotonicity Principle for the Dirichlet Energy in the nonlinear case. This includes the $p$-Laplacian \eqref{pOhm_potential} and linear \eqref{Ohm_potential} cases, where the proof is apparently simpler.
\begin{thm}\label{monoten}
Let $\Omega$ be an open bounded domain in $\R^n$ with Lipschitz boundary and ${\sigma}_1, \sigma_2$ satisfying (H1), (H2), (H3), (H4), (H5). Then, 
\begin{equation*}
\sigma_1\leq\sigma_2\quad\Longrightarrow\quad\mathbb F_{\sigma_1}(u_{1}^f) \leq \mathbb F_{\sigma_2} (u_{2}^f)\quad \forall f\in X_\diamond,
\end{equation*}
where $\sigma_1\leq \sigma_2$ is meant in the sense \eqref{defmon} and $u_{i}^f$ is the minimizer of \eqref{gminimum} with $\sigma=\sigma_i$, for $i=1,2$.
\end{thm}
\begin{proof}
Since $u_{2}^f$ is an admissible function for problem \eqref{gminimum} with $\sigma=\sigma_1$, we have
\begin{equation*}
\begin{split}
\mathbb{F}_{\sigma_1}( u_{1}^f)& \leq \mathbb{F}_{\sigma_1}(u_{2}^f)  =\int_\Omega Q_{\sigma_1} (x,|\nabla u_{2}^f(x)|)\ \text{d}x\\
& = \int_\Omega \int_{0}^{|\nabla u_{2}^f(x)|}\sigma_1\left( x,\xi\right) \xi\ \text{d}\xi\ \text{d}x \leq \int_\Omega \int_{0}^{|\nabla u_{2}^f(x)|} \sigma_2\left( x,\xi\right)\xi\  \text{d}\xi\ \text{d}x= \mathbb F_{\sigma_2}(u_{2}^f),
\end{split}
\end{equation*}
where the second inequality follows from the assumption $\sigma_1\leq\sigma_2$.
\end{proof}

\subsection{Connection between Dirichlet Energy and DtN operator}\label{connection}
The motivation for our research is in generalizing the Monotonicity Principle from linear and $p$-Laplacian to more general cases. The first step in this direction is to study the polynomial type nonlinear constitutive relationship. For this, let us consider
\[
\sigma(x, E)=\sum_{k=0}^N\theta _{k}\left( {x}\right) {E}^{k}\quad \text{for a.e.} \ x\in\overline\Omega\ \text{and}\ \forall E>0,  
\]
where either $\theta_k\in L_+^\infty(\Omega)$ or $\theta_k\equiv0$ in $\overline\Omega$ and $N\in\N$.
This polynomial type nonlinearity leads to minimization problem \eqref{gminimum} where $p=N+2$, $u\in W^{1, N+2}(\Omega)$, $u|_{\partial\Omega}=f\in X_\diamond$ and 
\begin{equation*}
\mathbb F_\sigma (u)= \sum_{k=0}^N\int_\Omega\left( \frac{\theta _{k}(x)}{k+2}|\nabla u (x)|^{k+2}\right) \text{d} x.
\end{equation*}
Furthermore, we have
\begin{equation}\label{poly_formDtN}
\left\langle \Lambda_\sigma \left( f\right),f\right\rangle=
\sum_{k=0}^N \int_{\Omega }\theta _{k}(x)\left\vert \nabla u (x)\right\vert ^{k+2}\text{d}x\quad \forall f\in X_\diamond,
\end{equation}
and, consequently,
\begin{equation*}\mathbb F_\sigma\left( u^f\right) \neq\left\langle \Lambda_\sigma \left( f\right),f\right\rangle\text{ in }  X_\diamond.
\end{equation*}
This is a major issue because it prevents to transfer the monotonic connection of electrical conductivity with the boundary data. 

Here, we investigate the issue which derives the proportionality between Dirichlet Energy and the power product \eqref{poly_formDtN}. This proportionality holds only for some special case of nonlinearity.  For a general nonlinear constitutive relationship satisfying (H1)-(H5), the Dirichlet Energy \eqref{genergy} and the 
ohmic power \eqref{g_formDtN} are proportional if and only if there exists $c>0$ such that 
\begin{equation}
\label{proportionality}
\left\langle \Lambda_\sigma  \left( f\right),f \right\rangle=c\mathbb F_\sigma \left( u^f\right)\quad \forall f\in X_\diamond,
\end{equation}
 that is
\begin{equation}
\label{proportionalityDD}
\int_{\Omega }\left[ J_\sigma(x, E(x))E(x)-c\int_{0}^{E(x)}J_\sigma \left( x, \xi \right) \text{d}\xi \right] 
\text{d}x=0,\quad \forall f\in X_\diamond.
\end{equation}%
A sufficient condition for \eqref{proportionalityDD} is%
\begin{equation*}
J_\sigma \left(\cdot, E \right) E=c\int_{0}^{E }J_\sigma
\left( \cdot,\xi \right) \text{d}\xi,
\end{equation*}%
that is%
\begin{equation*}
J_\sigma^{\prime }(\cdot, E)E =\left( c-1\right) J_\sigma(\cdot, E) .
\end{equation*}%
This latter ODE gives $J_\sigma \left(\cdot, E \right) =aE^{c-1}$, with $a\in\R$ i.e.
$\sigma \left(\cdot, E \right) =aE ^{c-2}$. 
Hence the proportionality \eqref{proportionality} is not expected for general nonlinear constitutive relationships except for monomial type nonlinearities.

\subsection{The power product for the average DtN 
operator}
We prove that the Dirichlet Energy \eqref{genergy} is transferred to a boundary measurement involving $\overline{\Lambda}_\sigma$, i.e. to the power product $\left\langle\overline{\Lambda}_\sigma \left( f\right) ,f \right\rangle$. 
\begin{thm}\label{transferthm}
Let $\Omega$ be an open bounded domain in $\R^n$ with Lipschitz boundary and ${\sigma}$ satisfying (H1), (H2), (H3), (H4), (H5). Then
\begin{equation*}\mathbb{F}_{\sigma}\left( u^f\right)=\left\langle\overline{\Lambda}_\sigma \left( f\right) ,f \right\rangle
\quad \forall f\in X_\diamond,
\end{equation*}
where $u^f$ is the minimizer of \eqref{gminimum}.
\end{thm}

\begin{proof}
For $\alpha\in [0,1]$
, we set \[
g(\alpha):=(\mathbb F_\sigma\circ\mathbb{U_\sigma})  (\alpha f).
\]
By Proposition \ref{gateauxprop}, since $\mathbb F_\sigma\circ\mathbb{U_\sigma}$ is G\^{a}teaux-differentiable, $g$ is differentiable.  By replacing $f$ and $\varphi$ with $\alpha f$ and $f$, we have 
\[
g'(\alpha)=\text{d}(\mathbb F_\sigma\circ\mathbb{U_\sigma})(\alpha f;f)=\langle \Lambda_\sigma (\alpha f),f\rangle.
\]
Eventually, by integrating over the interval $\left[0, 1\right]$, we obtain
\begin{equation*}
\mathbb{F}_\sigma (u^f)=(\mathbb F_\sigma\circ\mathbb{U_\sigma})  (f)=g(1)=\int_0^1 g'(\alpha)\text{d}\alpha=\int_0^1\langle \Lambda_\sigma (\alpha f),f\rangle\  \text{d}\alpha=  \left\langle\overline{\Lambda}_\sigma \left( f\right) ,f \right\rangle.
\end{equation*}
\end{proof}

\subsection{Monotonicity Principle for $\overline{\Lambda}_\sigma$}
The main result of this section is to derive Monotonicity Principle for the operator \eqref{average_flown}.
\begin{thm}\label{monothm}
Let $\Omega$ be an open bounded domain in $\R^n$ with Lipschitz boundary and ${\sigma}_1, \sigma_2$ satisfying (H1), (H2), (H3), (H4), (H5). Then
\begin{equation}
\label{m_charge}
\sigma_1\leq\sigma_2\quad\Longrightarrow\quad \left\langle\overline{\Lambda}_{\sigma_1}  \left( f\right) ,f \right\rangle
\leq \left\langle\overline{\Lambda}_{\sigma_2}  \left( f\right) ,f \right\rangle
\quad \forall f\in X_\diamond,
\end{equation}
where $\sigma_1\leq \sigma_2$ is meant in the sense \eqref{defmon}.
\end{thm}
\begin{proof}
Let $u_{i}^f$ be the unique solution of problem \eqref{gminimum} with $\sigma=\sigma_i$, $i=1,2$, and boundary value $f$.
By Theorems \ref{monoten} and \ref{transferthm}, we have
\begin{equation*}
\begin{split}
\left\langle\overline{\Lambda}_{\sigma_1}  \left( f\right) ,f \right\rangle= 
\mathbb{F}_{\sigma_1}( u^f_{1}) \leq
 \mathbb{F}_{\sigma_2}( u^f_{2})  =
 \left\langle\overline{\Lambda}_{\sigma_2}  \left( f\right) ,f \right\rangle
\end{split}
\end{equation*}
that proves \eqref{m_charge}. 
\end{proof}

If the constitutive relationship $J_\sigma=J_\sigma \left(x,E\right)$ is strictly decreasing with respect to $E$, then the Monotonicity is reversed, i.e.
\begin{equation*}
\sigma_1\leq\sigma_2\quad\Longrightarrow\quad \left\langle\overline{\Lambda}_{\sigma_1}  \left( f\right) ,f \right\rangle
\geq \left\langle\overline{\Lambda}_{\sigma_2}  \left( f\right) ,f \right\rangle\quad\forall f\in X_\diamond.
\end{equation*}
We stress that Theorem \ref{monothm} generalizes both the $p$-Laplacian and the linear cases. Indeed, in such cases, we have 
\begin{equation*}
\begin{split}
\left\langle   \Lambda _{\sigma_1 }\left( f\right) ,f \right\rangle=
p\left\langle\overline{\Lambda}_{\sigma_1}  \left( f\right) ,f \right\rangle \leq
p\left\langle\overline{\Lambda}_{\sigma_2}  \left( f\right) ,f \right\rangle= 
\left\langle \Lambda _{\sigma_2 }\left( f\right) , f \right\rangle \quad\forall f\in X_\diamond.
\end{split}
\end{equation*}

\section{Conclusions}\label{Concl}
In this paper, we proved Monotonicity Principle for inverse electrical conductivity problems where the conductor is modeled by a fully nonlinear constitutive relationship. Namely, we proved Monotonicity Principle where the \lq\lq standard\rq\rq\! DtN operator $\Lambda_\sigma$, required for $p$-Laplacian and linear cases, is replaced by the Average DtN operator $\overline{\Lambda}_\sigma$.
Specifically, we unveiled the \lq\lq mechanics\rq\rq of Monotonicity
by first recognizing that $\mathbb F_\sigma\circ\mathbb{U_\sigma}$ is monotonic with respect to the material property $\sigma$, even in the nonlinear case and, then, by transferring this Monotonicity to the new boundary operator $\overline{\Lambda}_\sigma$. This is a relevant result because, apart from linear and $p$-Laplacian cases, it is impossible to transfer the Monotonicity from $\mathbb F_\sigma\circ\mathbb{U_\sigma}$ to the DtN operator $\Lambda_\sigma$. This result is based on the fundamental relation $\text{d} (\mathbb F_\sigma\circ\mathbb{U_\sigma})=\Lambda_\sigma$ proved in Proposition \ref{gateauxprop}.  
During this analysis, we proved two general results: 
(i) the virtual power product $\langle\Lambda_\sigma (f),\varphi\rangle$ is equal to $\text{d} \mathbb F_\sigma(u^f;\varphi)$, that is the G\^{a}teaux derivative of the Dirichlet Energy $\mathbb{F}_\sigma$, evaluated at the solution $u^f$ in the direction of $f$, and (ii) the power product $\langle\Lambda_\sigma (f),f\rangle$ is equal to $\text{d} (\mathbb F_\sigma\circ\mathbb{U_\sigma})(f;f)$, that is, the G\^{a}teaux derivative of the composed function $\mathbb{F}_\sigma\circ \mathbb U_\sigma$, evaluated at $f$ in the direction of $f$.

Future work will deal with the adaptation of Monotonicity Principle in an inversion or imaging method for nonlinear materials.

\section*{Acknowledgements}
This work has been partially supported by the MiUR-Progetto Dipartimenti di eccellenza grant \lq\lq Sistemi distribuiti intelligenti\rq\rq of Dipartimento di Ingegneria Elettrica e dell'Informazione \lq\lq M. Scarano\rq\rq, by MiSE project \lq\lq SUMMa: Smart Urban Mobility
Management\rq\rq and by GNAMPA of INdAM.

\bibliographystyle
{iopart-num}
\bibliography{biblioCFPPT}
\end{document}